\documentclass[10pt]{amsart}

\usepackage{amssymb}
\usepackage[all]{xy}
\usepackage{enumerate}

\usepackage{bbm}

\usepackage[dvipsnames]{xcolor}
\usepackage{graphicx}

\usepackage{tikz-cd}

\newcommand{\rar}[1]{\stackrel{#1}{\longrightarrow}}

\newcommand{\eps}{\epsilon}

\newcommand{\bF}{{\mathbb F}}

\newcommand{\bZ}{{\mathbb Z}}

\newcommand{\cA}{{\mathcal A}}
\newcommand{\cB}{{\mathcal B}}
\newcommand{\cC}{{\mathcal C}}

\newcommand{\cF}{{\mathcal F}}

\newcommand{\cM}{{\mathcal M}}

\newcommand{\cO}{{\mathcal O}}

\newcommand{\cS}{{\mathcal S}}

\newcommand{\idot}{{\:\raisebox{1pt}{\text{\circle*{1.5}}}}}

\newcommand{\colim}{\text{colim}}

\newcommand{\uR}{\underline{R}}

\newcommand{\td}{\tilde{d}}

\newcommand{\tA}{\widetilde{A}}

\newcommand{\tQ}{\widetilde{Q}}
\newcommand{\tR}{\widetilde{R}}
\newcommand{\tX}{\widetilde{X}}

\newcommand{\ttA}{\breve{A}}

\newcommand{\ttd}{\breve{d}}

\newcommand{\nc}{\newcommand}

\nc\wh{\widehat}

\nc\on{\operatorname}

\nc\Gr{\on{Gr}}

\nc\Fl{\on{Fl}}

\newtheorem{thm}[subsection]{Theorem}

\DeclareMathOperator{\gr}{{gr}}

 \DeclareMathOperator{\Spf}{{Spf}}

\DeclareMathOperator{\Mod}{{Mod}}
\DeclareMathOperator{\DGA}{{DGA}}

\DeclareMathOperator{\hocolim}{{hocolim}}

\DeclareMathOperator{\Sp}{{Sp}}

\newcommand{\limto}{{\displaystyle\lim_{\longrightarrow}}}
\newcommand{\rightlim}{\mathop{\limto}}

\newcommand{\leftlim}{\mathop{\displaystyle\lim_{\longleftarrow}}}
\newcommand{\limfromn}{\leftlim\limits_{\raise3pt\hbox{$n$}}}
\newcommand{\limton}{\rightlim\limits_{\raise3pt\hbox{$n$}}}

\newcommand{\rightlimit}[1]{\mathop{\lim\limits_{\longrightarrow}}\limits%
                    _{\raise3pt\hbox{$\scriptstyle #1$}}}

\newcommand{\leftlimit}[1]{\mathop{\lim\limits_{\longleftarrow}}\limits%
                    _{\raise3pt\hbox{$\scriptstyle #1$}}}

\newcommand{\iso}{\buildrel{\sim}\over{\longrightarrow}}

\DeclareMathOperator{\Hom}{{Hom}}

 \DeclareMathOperator{\id}{{id}}
 
\DeclareMathOperator{\Mor}{{Mor}}

\DeclareMathOperator{\Perf}{{Perf}}
\DeclareMathOperator{\Spec}{{Spec}}

\DeclareMathOperator{\oHH}{{\overline{HH}}}
\DeclareMathOperator{\wHH}{{\widehat{\overline{HH}}}}
\DeclareMathOperator{\HH}{{HH}}
\DeclareMathOperator{\tC}{{\tilde {\mathcal C}}}
\DeclareMathOperator{\hHP}{{\widehat{HP}}}
\DeclareMathOperator{\HP}{{HP}}
\DeclareMathOperator{\TC}{{TC}}
\DeclareMathOperator{\wMod}{{\widehat{Mod}}}
\DeclareMathOperator{\cris}{{cris}}
\DeclareMathOperator{\zar}{{zar}}
\DeclareMathOperator{\ev}{{ev}}
\DeclareMathOperator{\cone}{{cone}}
\DeclareMathOperator{\Sh}{{Sh}}
\DeclareMathOperator{\PSh}{{PSh}}
\DeclareMathOperator{\cofib}{{cofib}}
\DeclareMathOperator{\fib}{{fib}}

\newcommand{\HKR}{\mathrm{HKR}}

\makeatletter

\newcommand{\Rmnum}[1]{\expandafter\@slowromancap\romannumeral #1@}
\makeatother

\newtheorem{Th}{Theorem}
\newtheorem{pr}{Proposition}[section]
\newtheorem{lm}[pr]{Lemma}

\newtheorem{cor}[pr]{Corollary}

\theoremstyle{definition}

\newtheorem{df}[pr]{Definition}
\newtheorem{rem}[pr]{Remark}

\numberwithin{equation}{section}

\newcommand{\TP}{\operatorname{TP}}
\newcommand{\hTP}{\operatorname{\widehat{TP}}}

\newcommand{\Fun}{\operatorname{Fun}}

\newcommand{\dR}{\mathrm{dR}}
\newcommand{\RG}{\mathrm{R}\Gamma}

\newcommand{\RGcr}{\mathrm{R}\Gamma_{\mathrm{cris}}}

\usepackage{relsize}
\usepackage[bbgreekl]{mathbbol}
\usepackage{amsfonts}
\DeclareSymbolFontAlphabet{\mathbb}{AMSb}
\DeclareSymbolFontAlphabet{\mathbbl}{bbold}
\newcommand{\Prism}{{\mathlarger{\mathbbl{\Delta}}}} 

\newcommand{\br}{W}

\begin{document}

\title[On the periodic topological cyclic homology]
{On the periodic topological cyclic homology of DG categories in characteristic p }




\author[A.~Petrov]{Alexander Petrov}
\address{Harvard University, USA}
\email{alexander.petrov.57@gmail.com}
\author[ V. ~Vologodsky]{ Vadim Vologodsky}
\address{National Research University ``Higher School of Economics'',  Russia}
\email{vologod@gmail.com}

\begin{abstract} 
We prove that the $p$-adically completed periodic topological cyclic homology of a DG category over a perfect field $k$ of characteristic $p>2$  is isomorphic 
 to the ($p$-adically completed)  periodic  cyclic homology of a lifting of the DG category over the Witt vectors $W(k)$.

 \end{abstract}

\maketitle

\tableofcontents

\section{Introduction}
Let $\cC$ be a DG category over a perfect field $k$ of characteristic $p$, and let $\TP(\cC)$ be the  periodic topological cyclic homology spectrum introduced by L. Hesselholt in (\cite{h}).
Denote by $\hTP (\cC)$ the $p$-adic completion of the spectrum $\TP(\cC)$. A lifting of $\cC$ over the ring $W(k)$ of Witt vectors is a DG category $\tC$  over $W(k)$ whose objects are those of $\cC$ and morphisms are complexes of flat $W(k)$-modules, together with quasi-isomorphisms 
$$  Mor_{\tC} (X,Y)\otimes _{W(k)} k   \iso  \Mor_\cC (X,Y),$$
for every pair of objects $X$ and $Y$, compatible  with the compositions. Denote by $\hHP(\tC/W(k))$ the $p$-adic completion of the periodic cyclic homology complex $\HP(\tC/W(k))$. By definition,
$\hHP(\tC/W(k))$ is an object of the derived category of $W(k)$-modules; abusing notations we shall also denote by $\hHP(\tC/W(k))$  the underlying spectrum.
The main result of this paper is the following.

\begin{Th}\label{mainthint} For any  DG category $\cC$ over a perfect field $k$ of characteristic $p>2$ and a lifting $\tC$ of $\cC$ over $W(k)$, one has a natural isomorphism of spectra
\begin{equation}\label{maintheq}
\hTP (\cC) \iso \hHP(\tC/W(k)).
\end{equation}
\end{Th}
 Recall that the  periodic topological cyclic homology spectrum of a DG category over $k$ is a module over the topological cyclic homology ring spectrum $\TC(k)$. Using a computation from
 (\cite{ns}), for every perfect field $k$, the connective cover of  $\TC(k)$ is identified with the Eilenberg-MacLane ring spectrum $HW(k)$. This makes $\hTP (\cC)$ into a $HW(k)$-module. Note that the right-hand side of (\ref{maintheq}) has  a $HW(k)$-module structure by construction. We do not know if isomorphism (\ref{maintheq}) can be promoted to an isomorphism
 of $HW(k)$-modules. However, it does induce an isomorphism of $W(k)$-modules:
 $$\pi_i \hTP (\cC) \iso \hHP _{i}(\tC/W(k)).$$
 
To explain the key idea in the proof of Theorem \ref{mainthint} we need to introduce a bit of notation.
Recall that for a DG category $\cC$ over a commutative ring $\br$ the Hochschild homology $\HH(\cC/\br)$ is naturally an object of the
 symmetric monoidal stable $\infty$-category  $\Mod_{\br[S^1]}$ of $\br$-modules equipped with an action of the circle. The  periodic cyclic homology  $\HP(\cC/W)$ is constructed from
 $\HH(\cC/\br)$  by applying the Tate invariants functor $ \Mod_{\br[S^1]} \to  \Mod_{\br^{tS^1}}$, $M \mapsto M^{tS^1}$,  with respect to the circle action.
 We consider 
  a certain quotient of $\Mod_{\br[S^1]}$ by a monoidal ideal the  Tate invariants functor factors through:
\begin{equation}\label{tatecatint}
 \Mod_{\br[S^1]} \to  \Mod_{\br[S^1]}^t  \to  \Mod_{\br^{tS^1}}.
\end{equation}
The second arrow in (\ref{tatecatint}) is fully faithful when restricted to the subcategory of  $\Mod_{\br[S^1]}^t$ spanned by $W$ (with the trivial circle action).  Moreover, we have that, for every $M\in   \Mod_{\br[S^1]}$
$$ M ^{tS^1} \iso Mor_{  \Mod_{\br[S^1]}^t  }(W, M).$$
The advantage of  $ \Mod_{\br[S^1]}^t$ is that the functor $ \Mod_{\br[S^1]} \to  \Mod_{\br[S^1]}^t$ is symmetric monoidal, whereas    $ \Mod_{\br[S^1]} \to  \Mod_{\br^{tS^1}}$  is not. 

We also consider the $p$-completed version of (\ref{tatecatint}). For a fixed prime number $p$ and  a  stable $\infty$-category $\cM$ we denote by $\widehat{\cM}$ the  stable $\infty$-category obtained from $\cM$ by first taking the $p$-completions of 
the spaces of morphisms in $\cM$ and then applying the stabilization.  We apply this construction  to (\ref{tatecatint}).
\begin{equation}\label{tatecatintp-compl}
 \widehat{\Mod}_{\br[S^1]} \to \widehat{\Mod}^t_{\br[S^1]}   \to \widehat{\Mod}_{\br^{tS^1}}.
\end{equation}
Denote by  $\wHH(\cC/  \br)$ the image of   $\HH(\cC/\br)$ in the category  $  \widehat{\Mod}^t_{\br[S^1]}$.
For any $E_\infty$-algebra $\cA$ over $W$, the object  $\wHH(\cA/\br)  \in  \widehat{\Mod}^t_{\br[S^1]}$ has a natural structure of an  $E_\infty$-algebra over $\wHH(\br/\br)$.

The key step in the proof of  Theorem \ref{mainthint} is a construction
of a homomorphism of $E_\infty$-algebras over $\wHH(W(k)/W(k))$
\begin{equation}\label{augmentation.Intro}
 \wHH(k/W(k)) \rar{} \wHH(W(k)/W(k)),
\end{equation}
that reduces to  $\oHH(k\otimes_ {W(k)} k/k )\to  \oHH(k/k)$ modulo $p$\footnote{Throughout this paper $\HH(\cA/W)$ stands for the ``derived'' version of the  Hochschild homology functor (see, for example, \cite[\S 2.2]{bms}).
In particular,  $\HH(k/W(k))$ is computed via the cyclic bar construction applied to a $W(k)$-flat replacement for $k$. }.
Our construction of  (\ref{augmentation.Intro})  uses the Gauss-Manin connection on the periodic cyclic homology. Namely,  given any $E_{\infty}$-algebra $\cA$ over the polynomial algebra
$W(k)[x]$
   we show the existence of a natural isomorphism of $E_{\infty}$-algebras 
\begin{equation}\label{GMC: intro}
  {\wHH} ( \cA_{x=0} /  W(k)) \iso {\wHH}( \cA_{x=p} /  W(k)).
  \end{equation}
Here $ \cA_{x=t}$ stands for the (derived) fiber $\cA \otimes_{W(k)[x]} W(k)[x]/(x-t)$ of $\cA$ over the point $x=t$. We then apply (\ref{GMC: intro}) to the DG algebra 
$\cA= W(k)[x,\epsilon]$ generated over $W(k)[x]$ by $\epsilon$ with $\deg \epsilon = -1$, $\epsilon^2 =0$, and $d \epsilon =x$. Observing that  $\cA_{x=p} \iso k$ and that
 the fiber $\cA_{x=0}\iso W(k)[\epsilon]$ admits an algebra map to $W(k)$ we get the augmentation  (\ref{augmentation.Intro}). Note that the classical construction of the
  Gauss-Manin connection on the periodic cyclic homology due to Daletsky, I. Gelfand, Tsygan, and Getzler (see \cite{dgt}, \cite{g}) is not quite sufficient for our purposes
  as it does not, at least  on the nose,  give an isomorphism of $E_{\infty}$-algebras (\ref{GMC: intro}); in \S \ref{gmcr} we propose a different construction of the connection where its multiplicative property is obvious. This is the technical heart of the paper. 
 
 Next, using the homomorphism of  $E_\infty$-algebras (\ref{augmentation.Intro}) we  define a lax symmetric monoidal functor
\begin{equation}
\HP ^{cris}(-  ,W(k)): \text{DG categories over $k$} \rar{}  \Mod_{W(k)^{tS^1}}
 \end{equation}
 as follows. Given a  DG category $\cC$ over $k$ we consider the  $\wHH(k/W(k))$-module  $\wHH(\cC/W(k))$. Using 
  (\ref{augmentation.Intro})  we form the ``tensor product''
  \begin{equation}
 {\wHH}(\cC/W(k)) \otimes _ {   {\wHH}(k/W(k))  } {\wHH}(W(k)/W(k)) \in \Fun(\Delta^{op},  \widehat{\Mod}_{W(k)[S^1]}^t)
   \end{equation}
  considered as a simplicial object in  $\widehat{\Mod}_{W(k)[S^1]}^t$  given by  the two-sided bar construction. Projecting this object to $ \Fun(\Delta^{op},   \widehat{\Mod}_{W(k)^{tS^1}})$ and taking the colimit over $\Delta^{op}$ we get
$  \HP ^{cris}(\cC ,W(k))$.

In Lemma \ref{finhh} we show that under some finiteness condition the construction of $\HP^{cris}$ can be performed directly in the category $\Mod_{W(k)^{tS^1}}$. Namely, assume that 
for a DG category $\cC$ there exists an integer $N$ such the Hochschild homology groups $H^i(\HH(\cC/k))$ vanish for $i<N$ . Then we have a quasi-isomorphism 
$$\HP^{cris}(\cC, W(k))\simeq \HP(\cC/ W(k))\widehat{\otimes}_{\HP(k/ W(k))}\HP(W(k)/W(k)).$$

We remark that, by construction, we have a canonical map of  $\HP(W(k)/W(k))$-modules  $ \widehat{\HP} (\cC/W(k)) \to   \HP ^{cris}(\cC,W(k))$. 
We show that, for any lifting $\tC$ of $\cC$, the compositions
  \begin{equation}\label{liftsplitting: intro}\HP (\tC /W(k)) \to   \HP (\cC  /W(k))\to  \HP ^{cris}(\cC ,W(k))\end{equation}
   \begin{equation}\label{TPsplitting: intro} \TP (\cC)  \to   \HP (\cC  /W(k))\to  \HP ^{cris}(\cC ,W(k))\end{equation}
  are  isomorphisms after the $p$-completion.

 Our definition  of the functor $  \HP ^{cris}(- ,W(k))$ is inspired by Bhatt's construction of crystalline cohomology from the derived de Rham cohomology (\cite{b}, Corollary 8.6).
 
\smallskip
We were informed that Peter Scholze has also obtained a proof  of  Theorem \ref{mainthint}  in case when $\cC$ is a smooth and proper DG category over $k$.
We are grateful to Alexander Beilinson, Bhargav Bhatt,  Alexander Efimov, Jacob Lurie, and Nick Rozenblyum  for useful discussions. Special thanks go to
Dmitry Kaledin for his constant attention to this work and many useful suggestions and to  Akhil Mathew for drawing our attention to Bhatt's paper \cite{b}. This work was completed while the first-named author was supported by a Clay Research Fellowship. The work of the second-named author was supported in part by RNF grant  N\textsuperscript{\underline{o}}~$18-11-00141$.


\section{The Gauss-Manin connection revisited}\label{gmcr}
We start by introducing a bit of notation. For a commutative ring $\br$   we denote by $\Mod_{\br[S^1]}$ 
the symmetric monoidal stable $\infty$-category of $\br$-modules equipped with an action of the circle, that is the category of functors from $BS^1$ to $\Mod_{\br}$. Let $s:*\to BS^1$ be the inclusion of the unique $0$-cell and consider the induction functor $s_{!}:\Mod_{\br}\to \Mod_{\br[S^1]}$.  Let $T \subset \Mod_{\br[S^1]}$ 
be the smallest stable  subcategory $T \subset \Mod_{\br[S^1]}$ that contains all objects of the form  $s_{!}N$, where 
$N$ is a $\br$-module. Note that, for any $M\in \Mod_{\br[S^1]}$, we have that
$$s_{!}(N\otimes_{\br} s^*M)=s_{!}N\otimes_{\br} M$$ In particular, we have  $s_!N \otimes_{\br}  M\in T$. It follows that $T$ is a tensor ideal in  $ \Mod_{\br[S^1]}$, that is, for every $M \in   \Mod_{\br[S^1]}$ and $M'\in T$,  the tensor product $M\otimes M'$ is in $T$.
 We can thus form the symmetric monoidal quotient category  $\Mod_{\br[S^1]}/T$, that we denote by $\Mod_{\br[S^1]}^t$. It is equipped with a symmetric monoidal projection functor
$$ \Mod_{\br[S^1]} \to \Mod_{\br[S^1]}^t, \quad M\mapsto \overline M $$

We observe that the Tate invariants functor 
$ \Mod_{\br[S^1]} \to  \Mod_{\br^{tS^1}}$ factors uniquely through 
\begin{equation}\label{GMC: TateInv}
\Mod_{\br[S^1]}^t \to  \Mod_{\br^{tS^1}}. 
\end{equation}
Abusing notation, we shall write $ \overline M ^{tS^1}$ for the latter applied to  $ \overline M  \in  \Mod_{\br[S^1]}^t$.
By functoriality, for every $ \overline M  \in  \Mod_{\br[S^1]}^t$, we have a morphism
$$ \Hom_{  \Mod^t_{\br[S^1]}}(\overline{\br}, \overline M) \to \Hom_{\Mod_{\br^{tS^1}}}(W^{tS^1},\overline{M}^{tS^1})= \overline M ^{tS^1}.$$
\begin{lm}
 The above morphism is quasi-isomorphism. 
\end{lm} 

\begin{proof} The proof is analogous to that of Lemma I.3.8(iii) of \cite{ns}. We have $\Hom_{\Mod^t_{\br[S^1]}}(\overline{\br},\overline{M})=\hocolim_{N\in T/M}\Hom_{\br[S^1]}(\br, \cofib(N\to M))=\hocolim_{N\in T/M}\cofib(N^{hS^1}\to M^{hS^1})=\hocolim_{N\in T/M}\fib(M^{tS^1}\to\cofib(N_{hS^1}[2]\to M_{hS^1}[2]))$. In the last equality we used that $N^{tS^1}$ vanishes and that $M^{tS^1}$ is the cofiber of the norm map $M_{hS^1}[1]\to M^{hS^1}$. Since $(-)_{hS^1}$ commutes with colimits this is equivalent to $\fib(M^{tS^1}\to (\hocolim_{N\in T/M}(N\to M))_{hS^1})$ and this fiber is just $M^{tS^1}$ because any $M$ is equivalent to $\colim_{N\in T/M}N$. 
\end{proof}

It follows from the Lemma that  functor (\ref{GMC: TateInv}) is fully faithful when restricted to the subcategory of  $\Mod_{\br[S^1]}^t$, whose objects are perfect as $\br$-modules.

Let $R$ be the divided power envelope of the ideal $(x)$ in the polynomial algebra $\br[x]$, $I\subset R$ the  divided power ideal generated by $x$.
For a DG algebra  $\cA \in \DGA_{R}$ and an integer $n\geq 0$, we set 
$$\cA_n:= \cA  \otimes_{R}  R/I^{[n+1]}  \in \DGA_{ R/I^{[n]}}.$$

The main result of this section is the following.  
\begin{Th}\label{GMC: Th}  For any $\cA \in   \DGA_{R}$,
the projection 
\begin{equation}\label{GMC: proj}
 \oHH ( \cA_n /  R/I^{[n+1]}) \to \oHH( \cA_0 /  \br)
 \end{equation}
admits a right inverse:
$$\alpha_n: \oHH( \cA_0 /  \br) \to  \oHH ( \cA_n /  R/I^{[n+1]}), $$
that induces a quasi-isomorphism 
$$ \oHH( \cA_0 /  \br) \otimes _{\br}  R/I^{[n+1]} \iso \oHH ( \cA_n /  R/I^{[n+1]}).$$
Moreover, the maps $\alpha_n$, $n=0, 1, \cdots$,    can be lifted to a morphism of symmetric monoidal functors from the category   $\DGA_{R}$ to 
the category of pro-objects in  $ \Mod_{\br[S^1]}^t$:
  $$\alpha:  \oHH( \cA_0 /  \br) \to  \lim\limits_{n}\oHH ( \cA_n /  R/I^{[n+1]}), $$
\end{Th}
\begin{rem}
Merely the existence of a right inverse to  (\ref{GMC: proj})
 can be easily derived from the existence of the Gauss-Manin connection on the periodic cyclic homology. Our main observation in Theorem \ref{GMC: Th} is that morphisms $\alpha_n$ behave well with  respect to the tensor product of algebras. In particular, if $\cA$ is an $E_m$ algebra over  $R$ then  (\ref{GMC: proj}) admits a right inverse as a morphism of 
 $E_{m-1}$ algebras.
\end{rem}

\begin{proof}

For the construction of the inverse we will work in the additive symmetric monoidal category of cyclic $\br$-complexes. That is, the category of functors from the Connes' cyclic category $\Lambda$ to the category of chain complexes of $\br$-modules. The objects of $\Lambda$ are indexed by positive integers and are denoted by $[1],[2],\dots$. See Appendix B in \cite{ns} or Chapter 6 of \cite{l} for a detailed discussion of the cyclic category.

Recall the construction of a particular cyclic $\br$-module whose corresponding $S^1$-module will be free. Let $Q$ be the cyclic $\br$-module spanned by the cyclic set represented by $[1]$: $$Q([k])=\br\cdot \Hom_{\Lambda}([1],[k])$$

Given a DG algebra $\cB$ over a commutative ring $S$ we will denote by $\cB^{\#/S}$ the cyclic bar construction given by $[k]\mapsto \cB^{\otimes_S k}$(with the face maps defined using multiplication on $\cB$). If the base ring $S$ is $\br$ we will denote this cyclic object just by $\cB^{\#}$. The constant cyclic algebra $S^{\#/S}$ will be denoted by $\underline{S}$. Given two cyclic modules $X_1,X_2$ we define their tensor product object-wise: $(X_1\otimes_{\br}X_2)([k])=X_1([k])\otimes_{\br}X_2([k])$.

There is a functor of geometric realization from the category of cyclic modules to the category of modules with $S^1$-action:

\begin{pr}\label{cycreal}
There is a symmetric monoidal functor $$|-|:Func(\Lambda, \Mod_{\br})\to \Mod_{\br[S^1]}$$ such that

{\normalfont (i)} For any DG algebra $\cB$ over a commutative $\br$-algebra $S$ that is $h$-flat as an $S$-module there is a canonical quasi-isomorphism $|\cB^{\#/S}|\simeq \HH(\cB/S)$.

{\normalfont (ii)} The object $|Q|$ is quasi-isomorphic to $s_!M$ for a certain $\br$-module $M$.
\end{pr}

\begin{proof}

The functor is defined in Proposition B.5 of \cite{ns} and (i) is the definition of Hochschild homology. Denote by $\tQ$ the functor from the paracyclic category $\Lambda_{\infty}$ to the category of $\br$-modules spanned by the representable functor $\Hom_{\Lambda_{\infty}}([1]_{\Lambda_{\infty}},-)$. The projection $j:\Lambda_{\infty}\to \Lambda$ provides a functor $j_{!}:Func(\Lambda_{\infty},\Mod_{\br})\to Func(\Lambda,\Mod_{\br})$ that preserves colimits and takes $\tQ$ to $Q$. Hence, $|Q|=\colim_{\Lambda_{\infty}^{op}}{j^*j_!\tQ}=s_!(\colim_{\Lambda_{\infty}^{op}}\tQ)$ so the assertion (ii) is proven.
\end{proof}

The cyclic commutative algebra $R^{\#}$ contains the ideal $J:=\ker(R^{\#/\br}\to R^{\#/R}=\underline{R})$. It is equipped with the divided power structure coming from the inclusions $J([k])\subset\displaystyle\sum\limits_{i=1}^{k}  R^{\otimes (k-i)}\otimes I\otimes R^{\otimes (i-1)}$. Explicitly, $R^{\#}([k])$ is isomorphic to $\br[x_0^{[\cdot]},\dots x_{k-1}^{[\cdot]}]$ with $J([k])$ generated by elements of the form $x_i^{[l]}-x_j^{[l]}$ and the divided powers $(x_i^{[l]}-x_j^{[l]})^{[m]}$ are defined according to the usual binomial formula. 

Taking the quotient by the divided power square of the ideal $J$ we obtain the cyclic algebra $F$ given by $$F=R^{\#}/J^{[2]}$$ Note that the map $R^{\#}\to \underline{\br}$ induced by the quotient $R\to R/I= \br$ factors through a surjection $F\to \underline{\br}$. For brevity, denote the quotient $R/I^{[n+1]}$ by $R_n$.

\begin{lm}\label{freefil}{\normalfont (i)} For every $k$ the elements $1$ and $x_i^{[m]}$ for $i\in\{0,1,\dots k-1\},m\in\bZ_{\geq 1}$ form a $\br$-basis of the module $F([k])$.

{\normalfont (ii)} There exists a filtration on $F$ by $R^{\#}$-submodules $F=FIl^0\supset \ker(F\to\underline{\br})=Fil^1\supset Fil^2\supset\dots$ such that for every $i\geq 1$ the quotient $Fil^i/Fil^{i+1}$ is isomorphic to $Q$ as an $R^{\#}$-module(where $R^{\#}$ acts on $Q$ through the map $R^{\#}\to \underline{W}$). Moreover, for every $n$ the map $F\to \underline{R_n}$ factors through $F/Fil^{n+1}$.

{\normalfont (iii)} Denote $F/Fil^{n+1}$ by $F_n$. Consider the tensor product $F_n\otimes_{\br}\underline{R}_n$ as a module over $\underline{R}_n^{\#}$ through the action of $R_n^{\#}$ on $F_n$. The kernel of the multiplication map $F_n\otimes_{\br}\uR_n\to\uR_n$ admits a finite $R_n^{\#}$-linear filtration $\widetilde{Fil}^{\bullet}$ such that all quotients $\widetilde{Fil}^i/\widetilde{Fil}^{i+1}$ are isomorphic to $Q$ as $R_n^{\#}$-modules.

\end{lm}
\begin{proof}


(i) Firstly, we show that these elements indeed generate all of $F([k])=R^{\otimes k}/J([k])^{[2]}$ over $\br$. For any $i,j\in\{0,\dots, k-1\}$ we have $x_j^{[m+1]}=(x_i+x_j-x_i)^{[m+1]}=x_i^{[m+1]}+x_i^{[m]}(x_j-x_i)$ where the last equality holds because all other terms in the binomial formula are divisible by $(x_j-x_i)^{[l]}$ for $l\geq 2$ and thus lie in the ideal $J([k])^{[2]}$. It follows that $x_i^{[m]}x_j=mx_i^{[m+1]}+x_j^{[m+1]}$. Next, for any numbers $l,r$ we get $x_i^{[l]}x_j^{[r]}=x_i^{[l]}(x_i+x_j-x_i)^{[r]}=\binom{r+l}{l}x_i^{[r+l]}+x_i^{[l]}x_i^{[r-1]}(x_j-x_i)=(1-r)\binom{r+l}{r}x_i^{[r+l]}+\binom{r+l-1}{l}x_i^{[r+l-1]}x_j$. Combining it with the computation in the previous sentence, we obtain 
$$x_i^{[l]}x_j^{[r]}=\binom{r+l}{l}x_i^{[r+l]}+\binom{r+l-1}{l}x_j^{[r+l]}$$

Applying this formula repeatedly, we can express any monomial as a linear combination of the separate powers $x_i^{[m]}$ so these powers indeed generate the module $F([k])$. 

Suppose that there is a non-trivial linear relation between these powers. For every $i\in\{0,\dots, k-1\}$ there is a map $R^{\otimes k}/J([k])^{[2]}\to R/I^{[2]}$ induced by $x_j\mapsto\delta_{ij}x$, so the constant and degree $1$ terms of a relation must vanish. Next, the differential operator $\Delta=\displaystyle\sum\limits_{i=0}^{k-1}\partial_{x_i}$ on $R^{\otimes k}$ kills all the elements of the form $(x_i-x_j)^{[r]}$, in particular preserves the ideal $J([k])^{[2]}$ and thus acts on $F([k])$. On the other hand, applying $\Delta$ to a non-trivial linear combination of the powers $x_i^{[m]}$ an appropriate number of times gives a relation with a non-trivial linear term which we have seen to be impossible.

(ii) Having constructed an explicit basis in $F([k])$ we define the desired filtration by $Fil^i([k])=\langle x_j^{[m]}|m\geq i,j\in\{0,1,\dots, k-1\}\rangle$. The quotient $Fil^i/FIl^{i+1}$ is isomorphic to $Q$ as a cyclic $\br$-module. Since the $R^{\#}$-module structure on this quotient factors through $R^{\#}\to \underline{\br}$ we have constructed a filtration with the desired properties. 

(iii)The elements $t_j^{[m]}\otimes t^{[l]}-1\otimes t^{[m]}\cdot t^{[l]}$ for $1\leq m\leq n,0\leq l\leq n,0\leq j\leq k-1$ form a basis in the kernel of the multiplication map $m:F_n\otimes_{\br}\uR_n\to \uR_n$. For an index $i$ define the $i$-th step of the filtration $\widetilde{Fil}^i$ as 
\begin{multline}\widetilde{Fil}^i([k])=\langle t_j^{[m]}\otimes t^{[l]}-1\otimes t^{[m]}\cdot t^{[l]}|l\geq \lfloor\frac{i}{n}\rfloor+1\text{ or }\\ l=\lfloor\frac{i}{n}\rfloor\text{ and }m\geq(i\text{ mod }n)+1,j\in\{0,1,\dots, k-1\}\rangle\end{multline}

These are $R_n^{\#}$-submodules and the quotients $\widetilde{Fil}^{i}([k])/\widetilde{Fil}^{i+1}([k])$ admit a basis $t_j^{[(i\text{ mod }n)+1]}\otimes t^{\lfloor\frac{i}{n}\rfloor}-1\otimes t^{[(i\text{ mod }n)+1]}\cdot t^{[\lfloor\frac{i}{n}\rfloor]},j\in\{0,1,\dots,k-1\}$ so the cyclic modules $\widetilde{Fil}^i/\widetilde{Fil}^{i+1}$ are all isomorphic to $Q$ and we get the filtration $\ker(F_n\otimes_{\br}\uR_n\xrightarrow{m}\uR_n)=\widetilde{Fil}^0\supset \widetilde{Fil}^1\supset\dots\supset \widetilde{Fil}^{n(n+1)}=0$ with the desired property.


\end{proof}
We will use the auxiliary objects $F_n=F/Fil^{n+1}$ to produce the required maps in the category $\Mod^t_{\br[S^1]}$. Replace $\cA$ by a semi-free resolution(according to Proposition 13.5 of \cite{d}) over $R$ so that $\cA$ is $h$-flat as an $R$-module. From the explicit description of $F$ provided by the proof of the lemma we also see that $F_n=F\otimes_{R^{\#}}R_n^{\#}$. These cyclic objects come with the maps $F_n\to\underline{R_n}\xrightarrow{\ev_0}\underline{\br}$. 
For every $n$ we have a diagram of cyclic $\br$-modules

\[
\begin{tikzcd}
& \cA_n^{\#}\otimes_{R_n^{\#}}F_n\arrow[ld, "\gamma_n"]\arrow[rd, "\beta_n"] & \\
\cA_n^{\#/(R/I^{[n]})}=\cA_n^{\#}\otimes_{R_n^{\#}}\underline{R_n} & & \cA_n^{\#}\otimes_{R^{\#}}\underline{\br}=\cA_0^{\#/\br}
\end{tikzcd}
\]

The map $\beta_n$ is surjective with the kernel given by $\cA_n^{\#}\otimes_{R_n^{\#}}Fil^1F_n$. The filtration $Fil^{\bullet}F_n$ constructed in Lemma \ref{freefil}(ii) induces a finite filtration on this tensor product with quotients $gr^i(\cA_n^{\#}\otimes_{R_n^{\#}}Fil^1F_n)$ isomorphic to $\cA_0^{\#}\otimes_{W^{\#}}Q$. According to Proposition \ref{cycreal}, under the functor from cyclic $\br$-modules to $\Mod_{\br[S^1]}$ any module admitting a finite filtration with quotients of the form $M\otimes_{\underline{W}}Q$ gets sent to an object in the subcategory $T$. Hence, the map $\beta_n$ turns into an equivalence in $\Mod^t_{\br[S^1]}$ and the desired splitting is given by $$\alpha_n=\gamma_n\circ \beta_n^{-1}:\oHH(\cA_0/\br)\to \oHH(\cA_n/R/I^{[n]})$$

To prove that these maps induce equivalences $\oHH(\cA_0/\br)\otimes_{\br}R_n\simeq \oHH(\cA_n/R_n)$ consider the following diagram where $\beta'_n$ is obtained from $\beta_n$ by taking the tensor product with $\uR_n$ and $\gamma'_n $ is induced by the map $F_n\otimes_{\underline{\br}}\uR_n\to \uR_n\otimes_{\underline{\br}}\uR_n\xrightarrow{m} \uR_n$

\[
\begin{tikzcd}[column sep={11em,between origins}]
& \cA_n^{\#}\otimes_{R_n^{\#}}F_n\otimes_{\underline{\br}}\uR_n\arrow[ld, "\gamma'_n"]\arrow[rd, "\beta'_n"] & \\
\cA_n^{\#/(R/I^{[n]})}=\cA_n^{\#}\otimes_{R_n^{\#}}\underline{R_n} & & \cA_n^{\#}\otimes_{R_n^{\#}}\underline{\br}\otimes_{\underline{\br}}\uR_n=\cA_0^{\#/\br}\otimes_{\underline{\br}}\uR_n
\end{tikzcd}
\]

The maps $\beta'_n, \gamma'_n$ become equivalences in the category $\Mod^t_{\br[S^1]}$. Indeed, the filtration on $F_n$ induces a finite filtration on $\ker(\beta'_n)$ with quotients isomorphic to $Q\otimes_{\br}\cA_0^{\#}\otimes_{\br}\uR_n$ so $\beta'_n$ becomes an equivalence in $\Mod^t_{\br[S^1]}$. Similarly, the filtration constructed in \ref{freefil}(iii) induces a filtration on $\ker(\gamma'_n)$ with quotients isomorphic to $\cA_0^{\#}\otimes_{\br}Q$.

Hence, the maps $\alpha_n$ induce equivalences $$\gamma'_n\circ (\beta'_n)^{-1}:\oHH(\cA_0/\br)\otimes_{\br}\uR_n\simeq \overline{|\cA_0^{\#}\otimes_{\underline{\br}}\uR_n|}\xrightarrow{\sim} \oHH(\cA_n/R_n)$$
\end{proof}
\begin{rem}
(i) This theorem and its proof are analogous to the Poincare lemma for the de Rham cohomology of commutative algebras. It asserts that given a split pd-nilpotent thickening of $\br$-algebras $A\to A/I$ for any smooth scheme $X$ over $A$ the relative de Rham cohomology complex $\Omega^{\bullet}_{X/A}$ is canonically quasi-isomorphic to the constant module $\Omega^{\bullet}_{X\times_A A/I/A/I}\otimes_{A/I}A$. The object $F$ is analogous to the absolute de Rham complex $\Omega^{\bullet,pd}_{X/A/I}$ modded out by the divided power relations of the form $d(x^{[n]})=x^{[n-1]}dx$. Note, however, that in the theorem above we are restricting to the case of a particular pd-thickening $W[x^{[\cdot]}]\to W$. Even though one can define the cyclic algebra $F$ for an arbitrary thickening, the Lemma \ref{freefil} is false in general.

(ii)A simpler form of the above computation is used in Kaledin's construction \cite{k1} of the Gauss-Manin connection on periodic cyclic homology.
\end{rem}

Let $k$ be a perfect field of characteristic $p>0$. Denote by   $\widehat{\Mod}^t_{W(k)[S^1]}$ the category whose objects are those of  $\Mod^t_{W(k)[S^1]}$ and whose space of morphisms is obtained by the $p$-completion of that in  $\Mod^t_{W(k)[S^1]}$. The image of an object $\oHH(-)$ in the completed category is denoted by $\wHH(-)$.

\begin{cor} Assume that $p$ is an odd prime. Then, for any $E_m$ algebra $\cA$ over $W(k)[x]$, $\infty\geq m>0$,   we have an isomorphism
\begin{equation}\label{GMC: pt}
  \wHH ( \cA_{x=0} /  W(k)) \iso \wHH( \cA_{x=p} /  W(k))
  \end{equation}
of $E_{m-1}$ algebras in $\widehat{\Mod}^t_{W(k)[S^1]}$ that is equal to identity modulo $p$.
\end{cor}

\begin{proof}
For an integer $n\geq 0$ applying the above theorem to the base ring $\br=W_n(k)$ and the algebra $\cA\otimes_{W(k)[x]}R$ (here $R$ is the divided power envelope of $(x)$ in $W_n(k)[x]$) gives an equivalence: $$\oHH(\cA_{x=0}\otimes_{W(k)}W_n(k)/W_n(k))\otimes_{W_n(k)}R\simeq \oHH(\cA_k/R/I^{[m+1]})$$  for any $m$. For a large enough $m$ the map $\ev_p:R\to W_n(k)$ induced by $W_n(k)[x]\xrightarrow{x\mapsto p}W_n(k)$ factors through $R/I^{[m+1]}$(namely, take $m$ such that $v_p(\frac{p^{m+1}}{(m+1)!})\geq n$: it exists by the assumption $p>2$). Since $\oHH(\cA_k/R/I^{[m+1]})\otimes_{R/I^{[m+1]},\ev_p}W_n(k)\simeq \oHH(\cA_{x=p}\otimes_{W(k)}W_n(k)/W_n(k))$ the above equivalence yields an equivalence $$\oHH(\cA_{x=0}/W(k))\otimes_{W(k)}W_n(k)\simeq\oHH(\cA_{x=p}/W(k))\otimes_{W(k)}W_n(k)$$ These isomorphisms form an inverse system for varying $n$ and passing to the inverse limit proves the corollary.
\end{proof}
\begin{cor}\label{GMC: augmentation}
Assume that $p$ is an odd prime. Then there exists 
a homomorphism of $E_\infty$-algebras over $\wHH(W(k)/W(k))$ 
\begin{equation}\label{eq: augmentation.GMC}
 \wHH(k/W(k)) \rar{} \wHH(W(k)/W(k)),
\end{equation}
that reduces to  $\oHH(k\otimes_ {W(k)} k/k )\to  \oHH(k/k)$ modulo $p$.
\end{cor}

\begin{proof}
Consider the commutative DG algebra $\cA=W(k)[x][\eps]$ over $W(k)[x]$ with $\deg\eps=-1$ and $d\eps=x$. The fibers $\cA_{x=0}$ and $\cA_{x=p}$ are isomorphic to $W(k)[\eps]$ with $d\eps=0$ and $k$, respectively. The previous corollary gives an isomorphism $\wHH(k/W(k))\simeq \wHH(W(k)[\eps]/W(k))$ and the desired map is the composition of this isomorphism with the map induced by $W(k)[\eps]\xrightarrow{\eps\mapsto 0}W(k)$.
\end{proof}

\begin{rem}
Both corollaries fail for $p=2$. For a counterexample consider the algebra $\bF_2$ over $\bZ_2$: by Proposition 2.12 in \cite{kn} the periodic cyclic homology $\HP_0(\bF_2/\bZ_2)$ is isomorphic to $\bZ_2[[y^{[\cdot]}]]/(y-2)$ (note that this ring is already $2$-adically complete). The map $\bZ_2[y]\xrightarrow{y\mapsto 2}\bZ_2$ does not extend to this ring because $v_2(\frac{2^{2^n}}{{2^n}!})=1$ for every $n$ and it is not hard to prove that the ring $\HP_0(\bF_2/\bZ_2)$ admits no homomorphism to $\bZ_2$.
\end{rem}
\section{Crystalline periodic cyclic homology. Proof of the main theorem}
We start by introducing a bit of notation. Given  an algebra object $A$ in a symmetric monoidal $\infty$-category $\cS$, a right $A$-module $M$, and a left  $A$-module $N$ we denote by
\begin{equation}
``M\otimes _A N \text{''} \in \Fun(\Delta^{op}, \cS)
\end{equation}
  the two-sided bar construction. This is a  simplicial object of $\cS$ that carries $[n]\in \Delta^{op}$ to  $M\otimes A^{\otimes ^n}\otimes  N \in \cS$.
  
 For a DG category $\cC$ over a perfect field $k$ of odd characteristic we consider the  simplicial object 
   \begin{equation}
``  \wHH(\cC/W(k)) \otimes _ {   \wHH(k/W(k))  }  \wHH(W(k)/W(k)) \text{''}  \in \Fun(\Delta^{op},  \wMod_{W(k)[S^1]}^t),
   \end{equation}
 where   $\wHH(k/W(k))$-module structure on   $\wHH(W(k)/W(k)) $ is given by the algebra homomorphism (\ref{eq: augmentation.GMC}).
  The $p$-completed Tate invariants functor induces a functor between the corresponding categories of simplicial objects:
 $$  \Fun(\Delta^{op},  \wMod_{W(k)[S^1]}^t) \rar{}  \Fun(\Delta^{op},  \wMod_{W(k)^{tS^1}}). $$
 \begin{df} Define the crystalline periodic cyclic homology $ \HP ^{cris}(\cC, W(k)) $ to be the $p$-completion of 
 $$ \hocolim_{\Delta^{op}} (( ``\wHH(\cC/W(k)) \otimes _ {   \wHH(k/W(k))  }  \wHH(W(k)/W(k)) \text{''})^{tS^1} ).$$
 We also set 
 $$\HP ^{cris}(\cC, W_n(k))   =   \HP ^{cris}(\cC, W(k))  \otimes _{W(k)} W_n(k).$$
 \end{df}
\begin{thm}\label{thmaincrys}  Let $k$ be a perfect field of characteristic $p>2$. Then the following holds.
\begin{itemize}
\item[(i)] For any DG  category $\tilde \cC$ over $W_n(k)$, we have a natural isomorphism
$$  \HP (\tilde \cC/W_n(k))\iso \HP ^{cris}(\cC, W_n(k)) $$
where $\cC:= \tilde \cC\otimes _{W_n(k)} k.$ 
\item[(ii)] For any DG  category  $\cC$ over $k$ and a lifting $\tilde \cC$  of $\cC$ over $W(k)$, we have a natural isomorphism
$$  \hHP (\tilde \cC/W(k))\iso \HP ^{cris}(\cC, W(k)). $$
\item[(iii)]  For any DG  category  $\cC$ over $k$, there is a natural isomorphism of spectra
$$\HP ^{cris}(\cC, W(k))  \iso \widehat{\TP} (\cC).$$
\end{itemize}
\end{thm}
\begin{proof}
We first prove part (i) for $n=1$.  Since the tensor product commutes with colimits,  we have that
  \begin{equation}\label{eq1proofcrys}
 \HP ^{cris}(\cC, W_1(k))=   \hocolim_{\Delta^{op}}( (``\oHH(\cC/W(k)) \otimes  _{\oHH(k/W(k))} \oHH(k/k)\text{''})^{tS^1}).
   \end{equation}
 Observe that $ \oHH(k/W(k)) $-module structures on  $ \oHH(k/k)$ and on $\oHH(\cC/W(k))$ lift to the obvious    $\HH(k/W(k)) $-module structures on  $\HH(k/k)$ and on $\HH(\cC/W(k))$.
 It follows that the right-hand side of (\ref{eq1proofcrys}) can be rewritten as 
 $$  \hocolim_{\Delta^{op}}( (``\HH(\cC/W(k)) \otimes  _{\HH(k/W(k))} \HH(k/k)\text{''})^{tS^1}).$$
 We have to construct an isomorphism between this colimit and $\HP (\cC/k).$
 Observe that 
  $$\hocolim_{\Delta^{op}}( ``\HH(\cC/W(k)) \otimes  _{\HH(k/W(k))} \HH(k/k)\text{''})\iso $$
$$  \HH(\cC/W(k)) \otimes  _{\HH(k/W(k))} \HH(k/k)\iso  \HH(\cC/k).$$ 
In general, the Tate invariants functor does not commute with colimits. However, we get a morphism:
 \begin{equation}\label{eq2proofcrys}
 \hocolim_{\Delta^{op}}( (``\HH(\cC/W(k)) \otimes  _{\HH(k/W(k))} \HH(k/k)\text{''})^{tS^1})\to  \HP (\cC/k).
   \end{equation}
 We would like to show that  (\ref{eq2proofcrys}) is an isomorphism. To do this we give another  description of this map.
 The simplicial object $``\HH(\cC/W(k)) \otimes  _{\HH(k/W(k))} \HH(k/k)\text{''}$ is given by
$$[n] \mapsto \HH(\cC \otimes _{W(k)}  k^{ \otimes ^n}  \otimes _  {W(k)}   k/k  ),$$
where   $k^{ \otimes ^n}  $ stands for the $n$-fold tensor product over $W(k)$
 Denote by 
 $$i_*:  \text{DG categories over }\,  k   \to  \text{DG categories over }\,  W(k), \quad \cC \mapsto \cC,$$
 $$i^*:    \text{DG categories over }\,  W(k)  \to  \text{DG categories over }\,  k , \quad \tC \mapsto \tC\otimes _{W(k)} k$$ 
 the pair of adjoint functors. Setting  
 $$\Phi:= i^* i_* :   \text{DG categories over }\,  k    \to   \text{DG categories over }\,  k  ,$$
  we identify the simplicial DG category 
 $$[n] \mapsto \cC \otimes _{W(k)}  k^{ \otimes ^n}  \otimes _  {W(k)}   k  $$
 with the standard simplicial object
 $$[n] \mapsto    \Phi^{n+1}(\cC)  $$
 associated to the comonadic structure on $\Phi$.
 By a general property of a pair of adjoint functors (Corollary 8.6.9 in \cite{w}) the natural morphism 
 \begin{equation}\label{standardsimplicial}
  \Phi^\idot (\cC) \to \cC
  \end{equation}
    to the constant simplicial DG category  induces  a homotopy equivalence 
               \begin{equation}\label{weibel}
               i_*\Phi^\idot (\cC) \to i_* \cC.
                \end{equation}
                                        In particular,  morphism (\ref{standardsimplicial}) is a quasi-isomorphism: 
                           $$\hocolim_{\Delta^{op}}  \Phi^\idot (\cC) \iso \cC.$$
  Morphism (\ref{eq2proofcrys}) is the composition                         
  \begin{equation}\label{comparisonmodpmapbis}
  \HP ^{cris}(\cC, W_1(k)) \cong \hocolim_{\Delta^{op}}  \HP( \Phi^\idot (\cC) /k) \to \HP(\hocolim_{\Delta^{op}}  \Phi^\idot (\cC)/k )\cong  \HP(\cC/k).
 \end{equation}
 To show that (\ref{comparisonmodpmapbis}) is an isomorphism  we need the following lemma. 
 \begin{lm}[{\cite[Theorem~3.4]{amn}}]\label{akhil} 
 The functor 
 $$\cF\circ \HP(- /k):  \text{DG categories over }\,  k \to  \Mod_{k} \rar{\cF}   \Sp ,$$
 where  $  \Mod_{k} \rar{\cF}   \Sp$ is the forgetful functor from the category of $k$-vector spaces to the category of spectra,
 factors through $i_* :   \text{DG categories over }\,  k   \to  \text{DG categories over }\,  W(k)$.
 In fact, for every DG category $\cC$  over $k$,   one has a functorial isomorphism in $\Sp$:
 $$\HP(\cC/k)\iso \TP(\cC)/p.$$
 \end{lm}
 We want to prove that  (\ref{comparisonmodpmapbis}) is an isomorphism. Since the forgetful functor $\cF$ is conservative and commutes with colimits it suffices to check that
  \begin{equation}\label{comparisonmodpmap2}
\hocolim_{\Delta^{op}}  \cF \circ \HP( \Phi^\idot (\cC) /k) \to  \cF \circ  \HP(\cC/k)
 \end{equation}
 is an isomorphism. Applying the functor $\TP(-)/p$ to  (\ref{weibel})  and using the Lemma we derive a homotopy equivalence 
  $$ \cF \circ \HP( \Phi^\idot (\cC) /k) \to  \cF \circ  \HP(\cC/k).$$ 
 This proves that  morphism (\ref{comparisonmodpmap2})
 is an isomorphism. 
 
 Now we sketch a proof of (i) for any $n$. Consider the $E_{\infty}$-algebras homomorphism
 \begin{equation}\label{comparisonmodpmap3}
  \wHH(W_n(k)/W(k))\rar{}  \wHH(W(k)/W(k)),
  \end{equation}
  defined as the composition 
  $$   \wHH(W_n(k)/W(k))\rar{}  \wHH(k/W(k))\rar{}      \wHH(W(k)/W(k)),  $$
 where the first map is induced by the projection $W_n(k) \to k$ and the second map is  (\ref{eq: augmentation.GMC}).
 Given a DG category $\tC$ over $W_n(k)$ we use (\ref{comparisonmodpmap3}) to  define  
  $ \HP ^{cris}(\tC, W(k))$
  as the $p$-completion   of
  $$ \hocolim_{\Delta^{op}} (( ``\wHH(\tC/W(k)) \otimes _ {   \wHH(W_n(k)/W(k))  }  \wHH(W(k)/W(k)) \text{''})^{tS^1} ).$$
  The  $W(k)$-linear functor $\tC \to \cC$ gives a morphism 
   \begin{equation}\label{comparisonmodpmap4}
   \HP ^{cris}(\tC, W(k)) \to  \HP ^{cris}(\cC, W(k)). 
   \end{equation}
  By looking at the reduction modulo $p$ one checks that (\ref{comparisonmodpmap4}) is an isomorphism. Next,  setting 
   $$ \HP ^{cris}(\tC, W_n(k)) =   \HP ^{cris}(\tC, W(k)) \otimes_{W(k)} W_n(k)$$
  we shall construct a morphism
   \begin{equation}\label{comparisonmodpmap5}
   \HP ^{cris}(\tC, W_n(k)) \to  \HP (\tC/ W_n(k)). 
   \end{equation}
  The construction is based on the following property of the homomorphism (\ref{comparisonmodpmap3}): its  reduction  modulo $p^n$ is given by
$$   \HH(W_n(k) \otimes_ {W(k)}    W_n(k) /W_n(k))\rar{}  \HH(W_n(k)/W_n(k)).$$
Using this we identify   $\HP ^{cris}(\tC, W_n(k)) $ with
 $$  \hocolim_{\Delta^{op}}(( ``\HH(\tC/W(k)) \otimes  _{\HH(W_n(k)/W(k))} \HH(W_n(k)/ W_n(k) )\text{''})^{tS^1}) .$$
which admits a map to $ \HP (\tC/ W_n(k))$.
 Finally, we need to check that (\ref{comparisonmodpmap5}) is an isomorphism. We do this using induction on $n$ and the following commutative diagram of $W_n(k)$-modules
  $$
\begin{tikzcd}
   \HP ^{cris}(\cC, W_1(k))  \arrow[r]\arrow[d] &  \HP ^{cris}(\tC, W_n(k)) \arrow[d]\arrow[r] &  \HP ^{cris}(\tC\otimes _{ W_n(k)}  W_{n-1}(k), W_{n-1}(k)) \arrow[d] \\
  \HP (\cC/ k)  \arrow[r] &   \HP (\tC/ W_n(k))     \arrow[r] &   \HP(\tC\otimes _{ W_n(k)}  W_{n-1}(k)/ W_{n-1}(k))  \\
\end{tikzcd}
$$
where the rows are  distinguished triangles. We have already proven that the left vertical arrow is an isomorphism. The right vertical arrow is an isomorphism by the induction assumption.
We conclude that (\ref{comparisonmodpmap5}) is an isomorphism. This completes the proof of part (i).

 The assertion of the part (ii) follows from (i). We, however, give a simpler proof which only uses (i) for $n=1$. 
 Consider the morphisms 
 \begin{equation}\label{splittingeqbis}
 \hHP (\tilde \cC/W(k))\rar{ }  \hHP (\cC/W(k))\rar{} \HP ^{cris}(\cC, W(k)),
 \end{equation}
 where the first arrow comes from $W(k)$-linear functor $ \tilde \cC \to \cC$ and the second arrow from the  homomorphism of $E_{\infty}$-algebras   $\wHH(k/W(k)) \rar{} \wHH(W(k)/W(k))$.
Modulo $p$ the composition reduces to an isomorphism 
$$\HP (\cC/k)\iso \HP ^{cris}(\cC, W_1(k))$$
  from part (i). Hence, the composition (\ref{splittingeqbis}) is an isomorphism.
  
For the assertion of part (iii) consider the arrows 
 $$\widehat{\TP} (\cC) \rar{} \hHP (\cC/W(k))   \rar{} \HP ^{cris}(\cC, W(k)) $$  
 where the second map is taken from  (\ref{splittingeqbis}). By Lemma \ref{akhil} the composition reduces to an isomorphism modulo $p$. Hence, it is an isomorphism.
 
\end{proof}

\begin{rem} Theorem \ref{thmaincrys} implies, in particular, that for any   DG  category $\tilde \cC$ over $W_n(k)$ and an integer $m\leq n$  the canonical map
 $$\HP (\tC/ W_n(k)) \otimes _{ W_n(k)}  W_{m}(k) \to  \HP(\tC\otimes _{ W_n(k)}  W_{m}(k)/ W_{m}(k)) $$
 is an isomorphism. This assertion is straightforward if $\tC$ satisfies the following boundedness  condition: $\HH_i(\tC \otimes _{ W_n(k)} k  /k )$ vanish for $i\gg 0$ or if $\tC$ admits a lift over 
 $W(k)$ but we do not have a direct proof in general.

\end{rem}

\begin{rem}
Theorem \ref{thmaincrys} implies that, for $p>2$, the $p$-completed periodic cyclic homology of a DG category $\tC$ over $W(k)$ only depends on the base change of the category to $k$. Arpon Raksit recently also proved this assertion by a different method and for all $p$, in the case when $\tC$ is the category of modules over a commutative smooth $W(k)$-algebra.
\end{rem}
\section{Explicit complex for $\HP^{cris}$}

Let $A$ be a DG algebra over $k$ such that the graded algebra $\bigoplus\limits_{i\in\bZ}A^i$ is a free associative graded algebra over $k$. In this section we construct an explicit complex representing $\HP^{cris}(A,W_2(k))$.

For a flat graded algebra $C$ over a commutative ring $S$ in this section we denote by $\HH(C/S)$ the graded module underlying the standard Hochschild complex of $C$ over $S$ and $b:\HH(C/S)\to \HH(C/S)[1], B:\HH(C/S)\to \HH(C/S)[-1]$ are the Hochschild and Connes-Tsygan differentials respectively so that $(\HH(C/S)((u)),b+uB)$ is the totalization of the periodic cyclic complex of $C$. For a degree $n$ derivation $d$ of the algebra $C$ we denote by $L_d$ the induced degree $n$ endomorphism of the Hochschild complex $\HH(C/S)$. In particular, if $C$ is a DG algebra with the differential $d$ then $(\HH(C/S)((u)),b+uB+L_d)$ is the periodic cyclic complex $\HP(C/S)$ of $C$.

The construction of the explicit complex is a formal consequence of the Cartan homotopy formula for periodic cyclic homology which we now recall:

\begin{lm}\label{cartanhtpy}
To any degree $n$ derivation $D$ of the graded algebra underlying a DG algebra $A$ over $k$ we can associate a degree $n-1$ endomorphism $\iota_D$ of the graded module $\HH(A,k)((u))$ such that the following relations are satisfied: $$[b+uB+L_d,\iota_D]=L_D+\iota_{[d,D]}\qquad \iota_{D_1+D_2}=\iota_{D_1}+\iota_{D_2}$$
\end{lm}

\begin{proof}
Take $\iota_D$ to be $e_D+u^{-1}E_D$ from Proposition 4.1.8 in \cite{l}. See also Definition 2.1 in \cite{g}.
\end{proof}

Pick a set $I$ of free generators for the graded associative algebra $\bigoplus\limits_i A^i=k\{x_i|i\in I\}$. Define a graded algebra $\bigoplus\limits_i\tA^i:=W_2(k)\{x_i|i\in I\}$ freely generated by the same set of generators so that there is an obvious isomorphism $(\bigoplus\limits_i\tA^i)\otimes_{W_2(k)}k\simeq\bigoplus\limits_i A^i$ of graded algebras. Let $\td:\tA^{\bullet}\to \tA^{\bullet+1}$ be an arbitrary lift of the differential $d:A^{\bullet}\to A^{\bullet+1}$ to a derivation of the graded algebra $\tA$. The operator $\td$ need not satisfy the relation $\td^2=0$ but there exists a unique $k$-linear map $D:A^{\bullet}\to A^{\bullet +2}$ such that $\td^2=pD$. Here $pD$ refers to the operator on $\tA$ induced by the composition $\tA^{\bullet}\to A^{\bullet}\xrightarrow{D}A^{\bullet +2}\xrightarrow{p}\tA^{\bullet +2}$.

\begin{lm}
The operator $D$ is a derivation of the DG algebra $A$.
\end{lm}

\begin{proof}
$D$ satisfies the Leibnitz rule because $\td$ does so. It also commuted with the differential $d$ because $pDd=\td^2d=\td^3=d\td^2=pdD$.
\end{proof}

Define the following complex
\begin{equation}\label{hpcrisnaive}\HP^{\cris}_{obj}(A,W_2(k)):=(\HH(\tA)((u)),b+uB+L_{\td}+p\iota_D)\end{equation}

The differential squares to zero because $[b+uB,p\iota_D]=pL_D=L_{[\td,\td]}=[L_{\td},L_{\td}]$. The next lemma shows that this complex does not depend on the choice of $\td$ up to a canonical isomorphism.

\newcommand{\tcC}{\tilde{C}}

\begin{lm}\label{indeplift}
Suppose that $\tcC$ is a flat graded algebra over $W_2(k)$ equipped with two degree $1$ derivations $\td_1,\td_2$ satisfying $\td_1\equiv\td_2\mod p$ and $\td_i^2\equiv 0\mod p$. Denote by $D_1,D_2$ the derivations of $C=\tcC\otimes_{W_2(k)} k$ such that $pD_i=\td_i^2$. There exists an isomorphism of complexes $$(\HH(\tcC)((u)),b+uB+L_{\td_1}+p\iota_{D_1})\simeq (\HH(\tcC)((u)),b+uB+L_{\td_2}+p\iota_{D_2})$$ which reduces to the identity modulo $p$.
\end{lm}

\begin{proof}
Let $D:C^{\bullet}\to C^{\bullet+1}$ be the degree $1$ derivation such that $pD=\td_1-\td_2$. The desired isomorphism is given by $\id+p\iota_{D}$. Indeed, $(\id+p\iota_D)(b+uB+L_{\td_1}+p\iota_{D_1})-(b+uB+L_{\td_2}+p\iota_{D_2})(\id+p\iota_D)=[p\iota_D,b+uB]+L_{\td_1}+p\iota_{D_1}+p\iota_{D}L_{\td_1}-L_{\td_2}-p\iota_{D_2}-pL_{\td_2}\iota_{D}$ Since $\td_1\equiv \td_2\mod p$, we have $p\iota_DL_{\td_1}-pL_{\td_2}\iota_D=[p\iota_D,L_{\td_1}]$ so the whole expression is equal to $[p\iota_D,b+uB+L_{\td_1}]+L_{\td_1}-L_{\td_2}+p\iota_{D_1-D_2}$. This is equal to zero by Lemma \ref{cartanhtpy} because $[\td_1,pD]=\td_1(\td_1-\td_2)+(\td_1-\td_2)\td_2=\td_1^2-\td_2^2=pD_1-pD_2$. 
\end{proof}

\begin{thm}\label{compnaive}
There is a canonical quasi-isomorphism $$\HP^{cris}_{obj}(A,W_2(k))\simeq \HP^{cris}(A, W_2(k))$$ where we regard the right-hand side as a DG module over $W_2(k)[u^{\pm 1}]$ using the equivalence of categories $\Mod_{W_2(k)^{tS^1}}\simeq \Mod_{W_2(k)((u))}$.
\end{thm}

\begin{proof}
First, the choice of a lift of $A$ over $W(k)$ clearly provides such quas-isomorphism as we can compute $\HP^{cris}_{obj}$ using the periodic cyclic complex of that lift, without any correction term.

To deduce the general case, for any DG algebra $A$ we will construct a splitting of the canonical map $\HP^{cris}_{obj}(A\otimes_{W(k)}^L k, W_2(k))\to \HP^{cris}_{obj}(A, W_2(k))$. We will use the following explicit resolution for the algebra $A$, viewed as an algebra over $W(k)$. Consider the free graded algebra $\ttA=W(k)\{x_i|i\in I\}$ over $W(k)$ and choose further a lift $\ttd:\ttA\to\ttA[1]$ of the derivation $\td$.  This yields a unique derivation $D:\ttA\to\ttA[2]$ such that $pD=\ttd^2$.

Consider the DG algebra $\ttA[\eps]$ whose underlying graded algebra is $\ttA\otimes_{W(k)}W(k)[\eps]$ with $\deg\eps=-1$ and the differential is given by $\ttd+p\partial_{\eps}+\eps D$. Here $\partial_{\eps}$ is the derivation induced by $W(k)[\eps]\to W(k)[1]$. The map $\ttA[\eps]\to A$ induced by the reduction map $\ttA\to A$ and $\eps\mapsto 0$ is a quasi-isomorphism of DG algebras over $W(k)$ (it is surjective and $\frac{1}{p}\eps$ is a contracting homotopy for the kernel).

In other words, $(\ttA[\eps], \ttd+p\partial_{\eps}+\eps D)$ is a lift of the algebra $A\otimes_{W(k)}^L k$ so $\HP_{obj}^{cris}(A\otimes^L_{W(k)}k,W_2(k))$ can be computed using the periodic cyclic complex of this lift: $$\HP(\tA[\eps]/W(k))=(\HH(\tA[\eps])((u)), b+uB+L_{\td}+pL_{\partial_{\eps}}+\eps  L_D)$$ Here $\tA$ denotes $\ttA/p^2$ and $\td$ is the reduction of $\ttd$. Apart from the contraction operators used in the definition of $\HP^{cris}_{obj}$, to construct the desired map we will need additional operations introduced in \cite{g}. Consider the operator $\sigma\{D,D\}$ defined in Section 3 of loc. cit. The operator $\rho\{D_1,D_2\}$ vanishes when $D_1,D_2$ are $1$-cocycles, so Lemma 3.2 of \cite{g} gives, in our particular case, the equality $[b+uB, u\iota_{D}L_D-\sigma\{D,D\}]=2\iota_D^2$.

We now define the map $\HP^{cris}_{obj}(A, W_2(k))\to \HP^{cris}_{obj}(A\otimes_{W(k)}^L k, W_2(k))$ by \begin{multline}\psi:=1+\eps\iota_D+\frac{p}{2}\eps(\sigma\{D,D\}-u\iota_D L_D): (\HH(\tA)((u)), b+uB+L_{\td}+p\iota_D)\to \\ (\HH(\tA[\eps])((u)), b+uB+L_{\td}+pL_{\partial_{\eps}}+\eps L_D)\end{multline}

This map indeed intertwines the differentials: $(1+\eps\iota_{D}+\frac{p}{2}\eps(\sigma\{D,D\}-u\iota_D L_D))(b+uB+L_{\td}+p\iota_D)-(b+uB+L_{\td}+pL_{\partial_{\eps}}+\eps L_D)(1+\eps\iota_{D}+\frac{p}{2}\eps(\sigma\{D,D\}-u\iota_D L_D))=[1+\eps\iota_D+\frac{1}{2}p\eps (\sigma\{D,D\}-u\iota_D L_D),b+uB+L_{\td}]+p\iota_D+p\eps\iota_D^2-\eps L_D-p\iota_D=0$

As remarked above, for the liftable algebra $A\otimes_{W(k)}^Lk$ the identification $\HP^{cris}_{obj}(A\otimes^L_{W(k)}k,W_2(k))\simeq \HP^{cris}(A\otimes^L_{W(k)},W_2(k))$ is immediate and, composing the map $\psi$ with it followed by the canonical map $\HP^{cris}(A\otimes^L_{W(k)}k,W_2(k))\to \HP^{cris}(A,W_2(k))$ we obtain a map $\HP^{cris}_{obj}(A,W_2(k))\to \HP^{cris}(A, W_2(k))$ whose base change to $k$ is the identity morphism, hence this map is an isomorphism.
\end{proof}
\begin{rem}
By Lemma 13.5 in \cite{d} any DG algebra $A$ is quasi-isomorphic to a DG algebra with free underlying graded algebra (furthermore, one can choose such semi-free model in a functorial way) so Theorem \ref{compnaive} provides an explicit complex computing $\HP^{cris}(A,W_2(k))$ for any DG algebra $A$ over $k$.
\end{rem} 

\section{ $\HP^{cris}$ of the category of perfect complexes and the crystalline cohomology} 

Let $X$ be a smooth variety over $k$.  The goal of this section is to establish the Hochschild-Kostant-Rosenberg style relation between $\HP^{cris}(\Perf(X),W(k))$ of the category of perfect complexes on $X$ and crystalline cohomology of $X$. Denote by $\RGcr(X/W(k))$ the crystalline cohomology of $X$ as an object of $D(W(k))$.

\begin{thm}\label{var: main}
 The object $\HP^{cris}(\Perf(X)/W(k))$ admits a canonical complete and exhaustive $\bZ$-indexed filtration $F_{\HKR}^i$ with graded pieces $\gr^i$ quasi-isomorphic to $\RGcr(X/W(k))[2i]$.
\end{thm}

\begin{rem}Given that $\HP^{cris}(\Perf(X)/W(k))$ is equivalent to $\TP(\Perf(X))$, a filtration with properties as in Theorem \ref{var: main} can also be constructed as the motivic filtration from Theorem 1.12 of \cite{bms}. The construction we give here is purely algebraic, and we expect the resulting filtration to coincide with the motivic filtration.
\end{rem}
To perform this comparison, recall the following construction of crystalline cohomology due to Bhargav Bhatt, cf. Corollary 8.6 in \cite{b}. This relationship between crystalline and derived de Rham cohomology motivated the construction of $\HP^{\cris}$ presented here -- this idea has been suggested to us by Akhil Mathew. Li and Liu have recently generalized this relationship to prismatic cohomology, their Theorem 3.5 \cite{liliu} implies Proposition \ref{var: der} below, we include another purely crystalline proof for the sake of completeness.

 For a commutative ring $A$ and a sheaf $\cO$ of commutative $A$-algebras on a topological space $T$ denote by $\dR_{\cO/A}$ the $p$-adic completion of Illusie's derived de Rham cohomology, as defined in \cite{i}, VII.2.1.1.2 or in \cite{b}, Definition 2.1. Denote also by $\widehat{\dR}_{\cO/A}\in D(T, A)$ the $p$-adic completion of the Hodge-complete derived de Rham cohomology \cite{i}, VII.2.1.3.3. If $\cO$ is annihilated by a power of $p$ then $p$-completion in the definition of $\widehat{\dR}_{\cO/A}$ is in fact redundant as graded pieces of the Hodge filtration are all $p$-complete.

\begin{pr}\label{var: der}
The commutative algebra ${\dR}_{k/W(k)}$ is concentrated in degree $0$ and the algebra map $W(k)\to {\dR}_{k/W(k)}$ admits a splitting that gives $W(k)$ the structure of a module over ${\dR}_{k/W(k)}$. There is a quasi-isomorphism of complexes of sheaves of $W(k)$-modules on $X$: \begin{equation}\label{derlocal}Ru_*\cO_{X/W,cris}\simeq {\dR}_{\cO_X/W(k)}\widehat{\otimes}^{L}_{{\dR}_{k/W(k)}} W(k)\end{equation} Left-hand side denotes the (automatically $p$-adically complete) pushforward of the crystalline structure sheaf along the morphism of sites $u:(X/W)_{\cris}\to X_{\zar}$. Moreover, passing to global section induces a quasi-isomorphism \begin{equation}\label{derRG}\RG_{cris}(X, W(k))\simeq \RG(X,{\dR}_{X/W(k)})\widehat{\otimes}^L_{{\dR}_{k/W(k)}}W(k)\end{equation}
\end{pr}

\begin{proof}

By Proposition 8.5 of \cite{b} we have $\dR_{k/W(k)}=\widehat{W(k)[x^{[\cdot]}]}/(x-p)$ ({\it loc. cit.} uses notation $\widehat{\dR_{B/A}}$ in place of ours $\dR_{B/A}$). The desired map $\dR_{k/W(k)}\to W(k)$ is given by $x^{[i]}\mapsto\frac{p^i}{i!}$. 

To prove the isomorphism \ref{derlocal} it is enough to construct a morphism \begin{equation}\label{derdesired}\dR_{\cO_X/W(k)}\to Ru_*\cO_{X/W(k),cris}\end{equation} such that its mod $p$ reduction is identified with the canonical map $\dR_{\cO_X\otimes_{W(k)}^L k/k}\to \dR_{\cO_X/k}\simeq \Omega^{\bullet}_{X/W(k)}\simeq Ru_*\cO_{X/W(k),cris}\otimes_{W(k)}k$. Indeed, precomposing such map with the morphism $\dR_{\cO_X/W(k)}\otimes_{\dR_{k/W(k)}}W(k)\to \dR_{\cO_X/W(k)}$ induced by $W(k)\to \dR_{k/W(k)}$ gives a morphism from RHS to LHS of \ref{derlocal} whose mod $p$ reduction is identified with the composition $\dR_{\cO_X\otimes^L_{W(k)}k/k}\otimes_{\dR_{k\otimes_{W(k)}^L k/k}}k\to \dR_{\cO_X\otimes^L_{W(k)}k/k}\to \dR_{\cO_X/k}$ that is a quasi-isomorphism by Kunneth formula.

The desired morphism is constructed in \cite{b}, Proposition 3.25 in the case of affine $X$. We will now spell out how to globalize this construction. Let $\cC_n$ be the category of triples $(U,T,U\to T)$ where $U$ is an open subscheme of $X$, $T$ is any {\it smooth} scheme over $W_n(k)$ and $U\to T$ is a morphism of $W_n(k)$-schemes. For a fixed open $V\subset X$ denote also by $\cC_n(V)$ the full subcategory of triples $(U,T,U\to T)$ with $U=V$. 

Denote by $\cF_1$ the presheaf with values in $D(W_n(k))$ on $\cC_n$ that assigns to $U\to T$ the de Rham cohomology $\RG_{\dR}(T/W_n(k))$ and let $\cF_2$ be the presheaf that assigns to $U\to T$ crystalline cohomology $\RG_{\cris}(U/W_n(k))$ of $U$. We get a morphism of presheaves $\cF_1\to\cF_2$ given by the composition $\RG_{\dR}(T/W_n(k))\simeq \RG_{\cris}(T/W_n(k))\to \RG_{\cris}(U/W_n(k))$ where the first identification uses smoothness of $T$.

There is a functor $\lambda:\cC_n\to X_{\zar}$ to the small Zariski site of $X$ that sends $U\to T$ to $U$. Denote by $\lambda_{!}:\Fun(\cC_n^{op},D(W_n(k)))\to \Sh(X_{\zar})$ the left Kan extension of presheaves along that functor composed with sheafification. We claim that the inverse limits of morphisms $\lambda_{!}\cF_1\to\lambda_{!}\cF_2$ for varying $n$ give the desired morphism \ref{derdesired}.

 The extension $\lambda_{!}\cF_1$ is identified with $\dR_{\cO_X/W_n(k)}$ by the definition of derived de Rham cohomology via left Kan extension from polynomial (or, equivalently, all smooth) algebras. To identify $\lambda_{!}\cF_2$ with $Ru_*\cO_{X/W_n(k), cris}$ it is enough to show that for a given affine $U$ the canonical map $\mathop{\mathrm{holim}}\limits_{\cC_n(U)}\RG_{cris}(U/W_n(k))\to \RG_{cris}(U/W_n(k))$ is an equivalence. 

Equivalently, we need to show that for any two objects $A,B\in D(W_n(k))$ the canonical map $\Hom(A,B)\to \Hom_{\PSh(\cC_n(U))}(\underline{A},\underline{B})$ is a homotopy equivalence of simplicial sets. Choose an object $(U\xrightarrow{i_0}T_0)\in\cC_n(U)$. Evaluation at that object induces a morphism $ \Hom_{\PSh(\cC_n(U))}(\underline{A},\underline{B})\to \Hom(A, B)$ and it remains to show that the composition $\Hom_{\PSh(\cC_n(U))}(\underline{A},\underline{B})\to \Hom(A, B)\to \Hom_{\PSh(\cC_n(U))}(\underline{A},\underline{B})$ is homotopic to identity. Any endofunctor $\alpha:\cC_n(U)\to\cC_n(U)$ induces a map $\alpha^*:\Hom_{\PSh(\cC_n(U))}(\underline{A},\underline{B})\to \Hom_{\PSh(\cC_n(U))}(\underline{A},\underline{B})$ and a natural transformation $f:\alpha\to\beta$ induces a homotopy from $\beta^*$ to $\alpha^*$. Consider the functor $\alpha_{U\to T_0}(U\xrightarrow{i} T):=(U\xrightarrow{i\times i_0} T\times T_0)$. The projections give natural transformations from $\alpha_{U\to T_0}$ to the identity endofunctor and to the constant endofunctor sending everything to $(U\to T_0)$. Thus, the above composition is homotopic to identity and the homotopy limit of the constant diagram over $\cC_n(U)$ is equivalent to the value of that diagram, as desired. 

This finishes the construction of the isomorphism (\ref{derlocal}), and (\ref{derRG}) follows by quasi-compactness of $X$.
\end{proof}

To utilize the above description of crystalline cohomology we need to know that under some finiteness condition the construction of $\HP^{cris}$ can be performed directly in the category $\Mod_{W(k)^{tS^1}}$.

\begin{lm}\label{finhh}
If  for a DG category $\cC$ there exists an integer $N$ such the Hochschild homology groups $H^i(\HH(\cC/k))$ vanish for $i<N$ then we have a quasi-isomorphism $$\HP^{cris}(\cC, W(k))\simeq \HP(\cC/ W(k))\widehat{\otimes}_{\HP(k/ W(k))}\HP(W(k)/W(k))$$
\end{lm}

\begin{proof}
The functor $\Mod_{W(k)[S^1]}^t\to \Mod_{W(k)^{tS^1}}$ from \ref{GMC: TateInv} is lax monoidal so there is a map 

$$\HP(\cC/ W(k))\widehat{\otimes}_{\HP(k/ W(k))}\HP(W(k)/W(k))\to \HP^{cris}(\cC, W(k))$$ It is enough to show that it is an equivalence after applying the tensor product $-\otimes_{W(k)^{tS^1}}k^{tS^1}$ which comes down to proving that the canonical map \begin{equation}\label{basechange}\HP(\cC/W(k))\otimes_{\HP(k/W(k))}\HP(k/k)\to \HP(\cC/k)\end{equation} is a quasi-isomorphism. For any complex $M$ equipped with $S^1$-action the Tate fixed points object $M^{tS^1}$ admits a $\bZ$-indexed filtration $F^i_{Tate}M^{tS^1}$ with quotients quasi-isomorphic to $M[2i],i\in\bZ$. Moreover, this filtration is exhaustive and $M^{tS^1}$ is complete with respect to it.

 It follows that \ref{basechange} can be upgraded to a map of $\bZ$-filtered objects.  Since the canonical map $\HH(\cC/W(k))\otimes_{\HH(k/W(k))}\HH(k/k)\to \HH(\cC/k)$ is a quasi-isomorphism, the map \ref{basechange} induces quasi-isomorphisms on the graded pieces of the filtrations. To conclude that this map is a quasi-isomorphism it is enough to see that the source is complete. This is the case because for any given $i$ the cohomology $H^i(F^j_{Tate}(\HP(\cC/W(k))\otimes_{\HP(k/W(k))}\HP(k/k)))$ vanishes for all large enough $j$.
\end{proof}


\begin{proof}({\it of Theorem \ref{var: main}})
For any qcqs scheme $Y$ over $W(k)$, Theorem 1.1 of \cite{a} equips $\hHP(Y/W(k))$ with a multiplicative filtration $F_{\HKR}^{*}\hHP(Y/W(k))$ such that $\gr^i_{F_{\HKR}}\simeq \widehat{\dR}_{Y/W(k)}[2i]$. If the derived de Rham complex $\widehat{\dR}_{Y/W(k)}$ happens to be concentrated in degree $0$ then this filtered object is canonically equivalent to $\hHP(Y/W(k))$ equipped with the double-speed Postnikov filtration: $F^i_{\HKR}\hHP(Y/W(k))\simeq \tau^{\leq -2i}\hHP(Y/W(k))$. This observation applies to $Y=\Spec k,\Spec W(k)$. In particular, the map $\hHP(k/W(k))\to \hHP(W(k)/W(k))$ provided by Corollary \ref{GMC: augmentation} can be automatically upgraded to a filtered morphism with respect to the filtrations $F_{\HKR}^{*}$. In what follows we use that $\hHP(X/W(k))$ defined via descent from affine opens (the object studied by \cite{a}) coincides with $\hHP(\Perf(X)/W(k))$, as follows from Corollary 5.2 of \cite{k}. 

This induces a filtration on the completed tensor product $\hHP(X/W(k))\widehat{\otimes}_{\hHP(k/W(k))}\hHP(W(k)/W(k))$. Namely, the simplicial bar construction \begin{equation}\label{var:hpbar}[n]\mapsto \hHP(X/W(k))\otimes_{W(k)}\hHP(k/W(k))^{\otimes n}\otimes_{W(k)}\hHP(W(k)/W(k))\end{equation} can be given the structure of a filtered simplicial $E_{\infty}$-algebra via the usual monoidal structure on the filtered derived category of $W(k)$-modules and the $p$-adic completion of the colimit of the bar construction is a filtered algebra with the underlying algebra $\HP^{cris}(X,W(k))=\hHP(X/W(k))\widehat{\otimes}_{\hHP(k/W(k))}\hHP(W(k)/W(k))$. In the last equality we used Lemma \ref{finhh} for $\cC=\Perf(X)$: we can take $N=-\dim X$ because of the HKR filtration on $\HH(X/k)$. Denote the resulting filtration on $\HP^{cris}(X, W(k))$ by $F_{\HKR}^{i}\HP^{cris}(X,W(k))$ for $i\in\bZ$.

The $d$-th graded piece $\cone(F_{\HKR}^{d+1}\to F_{\HKR}^d)$ is equivalent to the $p$-adic completion of the colimit of the simplicial diagram 
\begin{multline}\label{grbar1}[n]\mapsto \bigoplus\limits_{\substack{i_0,\dots, i_{n+1}\in\bZ \\ i_0+\dots +i_{n+1}=d}}\gr^{i_0}\hHP(X/W(k))\otimes \gr^{i_1}\hHP(k/W(k))\otimes\\
\dots\otimes \gr^{i_n}\hHP(k/W(k))\otimes_{W(k)}\gr^{i_{n+1}}\hHP(W(k)/W(k))\end{multline} 
with the morphisms induced by the maps of the form $\gr^iM_1\otimes \gr^j M_2\to \gr^{i+j}(M_1\otimes M_2)$.

Consider the category $E(\bZ)$ with objects given by the elements of $\bZ$ with exactly one morphism between every two objects. Let $N$ be the free simplicial $W(k)$-module generated by the simplicial nerve of $E(\bZ)$.  Since $\gr^i\hHP(X/W(k))\simeq \widehat{\dR}_{X/W(k)}[2i]$, the simplicial object (\ref{grbar1}) computing $\gr^d \HP^{cris}(X,W(k))$ is equivalent to the tensor product of $N$ with the bar construction computing the tensor product $(\RG(X,\widehat{\dR}_{X/W(k)})\otimes_{\widehat{\dR}_{k/W(k)}}W(k))[2d]$: \begin{equation}\label{grbar2}[n]\mapsto N_n\otimes_{W(k)}(\RG(X,\widehat{\dR}_{X/W(k)})\otimes_{W(k)}\widehat{\dR}_{k/W(k)}^{\otimes_{W(k)} n})[2d]\end{equation} where the identification takes the summand of (\ref{grbar1}) labeled by $i_0,i_1,\dots i_{n+1}$ to the submodule $W(k)\cdot [i_0\to i_0+i_1\to\dots\to i_0+i_1+\dots+i_{n}]\otimes \RG(\widehat{\dR}_{X/W(k)})\otimes\widehat{\dR}_{k/W(k)}^{\otimes_{W(k)} n}[2d]$

Since  $E(\bZ)$ is equivalent to the category with one object and one morphism, $N$ is homotopy equivalent to $W(k)$ and the graded pieces $\gr^d\HP^{cris}(X,W(k))$ are equivalent to $\RG_{cris}(X, W(k))[2d]$, as desired.

Finally, the above computation also shows that the canonical filtered map $\HP^{cris}(X,W(k))/p\simeq \HP(X/k)$ is a filtered quasi-isomorphism where $\HP(X/k)$ is equipped with the Hochschild-Kostant-Rosenberg filtration. In particular, the completeness of the filtration on $\HP^{cris}(X,W(k))$ follows from the corresponding result (Theorem 1.1 of \cite{a}) modulo $p$. The filtration on the non-completed colimit of (\ref{var:hpbar}) is exhaustive and, moreover, for any given $i\in \bZ$ the map $F^j_{\HKR}(\hHP(X/W(k)){\otimes}_{\hHP(k/W(k))}\hHP(W(k)/W(k)))\to \hHP(X/W(k)){\otimes}_{\hHP(k/W(k))}\hHP(W(k)/W(k))$ induces an isomorphism on $i$-th cohomology for all small enough $j$. Hence, the exhaustiveness persists after $p$-adic completion.
\end{proof}

Let us close with a discussion of the relationship between Proposition \ref{var: der} and Theorem 3.5 of \cite{liliu}. The theorem is stronger in that it moreover provides a splitting of the canonical morphism $\dR_{\cO_X/W(k)}\to Ru_*\cO_{X/W,cris}$. This relies on the additional structure presented by prismatic cohomology. Alternatively, if there is a lift of $X$ to a smooth formal scheme $\tX$ over $W(k)$ then such a splitting is given by $Ru_*\cO_{X/W,cris}\simeq \Omega^{\bullet}_{\tX/W(k)}=\dR_{\cO_{\tX}/W(k)}\to \dR_{\cO_X/W(k)}$. One can see that these splittings are in fact the same:

\begin{lm}
For an affine scheme $X=\Spec R$ and its lift $\tX=\Spf \tR$ over $W(k)$ there is a canonical identification

\begin{equation}
\begin{tikzcd}
{\Prism_{R/W(k)}^{(1)}} \arrow[rd] \arrow[dd,"\sim"] &    \\
              & {\dR_{R/W(k)}} \\
{\dR_{\tR/W(k)}} \arrow[ru] &   
\end{tikzcd}
\end{equation}

Here the vertical map is provided by the crystalline-de Rham comparison and the two oblique maps are the sections whose construction is outlined above.
\end{lm}

\begin{proof} On the one hand, there is a commutative diagram

\[
\begin{tikzcd}
\Prism^{(1)}_{R/W(k)}\arrow[r] & \dR_{R/W}\\
\Prism^{(1)}_{\tR/W(k)[[u]]}\arrow[r]\arrow[u] & \dR_{\tR/W(k)[[u]]}\arrow[u]
\end{tikzcd}
\]
that manifests functoriality of the identification constructed in Theorem 3.5 of \cite{liliu}. Combined with Corollary 3.17 of loc. cit. this provides us with the following commutative diagram 

\[
\begin{tikzcd}
                         & {\Prism^{(1)}_{R/W(k)}=\RG_{\cris}(R/W(k))} \arrow[rd] &    \\
{\Prism^{(1)}_{\tR/W(k)[[u]]}} \arrow[ru, "\text{crystalline specialization}"] \arrow[rd, "\text{de Rham specialization}"'] &               & {\dR_{R/W(k)}} \\
                         & {\dR_{\tR/W(k)}} \arrow[ru] &   
\end{tikzcd}
\]

Here the two maps on the right are the ones that we are looking to identify. This diagram allows us to do so because, by Corollary 3.7 of \cite{liliu}, both specialization maps factor through $\RG_{cris}(\tR/S)$ where $S$ is the divided power envelope of $(u-p)$ in $W(k)[[u]]$ and, by the base change for crystalline cohomology, there is a canonical trivialization of this module: $\RG_{cris}(\tR/S)\simeq \RG_{cris}(\tR/W(k))\otimes_{W(k)}S$. 
\end{proof}

The noncommutative situation is similar to the extent that a map $\HP^{cris}(\cC, W(k))\to \hHP(\cC/W(k))$ can be constructed either using a lift $\widetilde{\cC}$ (\ref{liftsplitting: intro}), if it exists, or using topological periodic cyclic homology (\ref{TPsplitting: intro}). However, we no longer know if these maps can be identified (in particular, we do not know whether different lifts might lead to different splittings).

\end{document}